\documentclass[a4paper,12pt,twoside]{amsart}

\usepackage{amsmath}
\usepackage{amsthm}
\usepackage{amssymb}
\usepackage[all]{xy}
\usepackage{hyperref}
\usepackage{rotating}
\usepackage[small, nohug, head=vee]{diagrams}
\diagramstyle[labelstyle=\scriptstyle]

\newtheorem{thm}{Theorem}[section]
\newtheorem{prop}[thm]{Proposition}
\newtheorem{lem}[thm]{Lemma}
\newtheorem{cor}[thm]{Corollary}
\newtheorem{conjecture}[thm]{Conjecture}

\numberwithin{equation}{section}

\theoremstyle{definition}
\newtheorem{definition}[thm]{Definition}
\newtheorem{remark}[thm]{Remark}

\newcommand{\Pic}{{\rm Pic}}

\renewcommand{\dim}{{\rm dim}\,}

\newcommand{\QQ}{\mathbb{Q}}
\newcommand{\RR}{\mathbb{R}}
\newcommand{\CC}{\mathbb{C}}

\begin{document}

\title{Varieties Fibered by Good Minimal Models}

\author[Ching-Jui\ Lai]{Ching-Jui Lai}

\address{C-J.L.: Department of Mathematics, University of Utah, 155 South 1400 East, Salt Lake City, UT 84112, USA}
\email{lai@math.utah.edu}



\begin{abstract}
Let $f:X\rightarrow Y$ be an algebraic fiber space such that the general fiber has a good minimal model. We show that if $f$ is the Iitaka fibration or if $f$ is the Albanese map with relative dimension no more than three, then $X$ has a good minimal model.
\end{abstract}

\maketitle

Two of the main conjectures in higher dimensional birational geometry are:
\begin{itemize}
\item {\bf Existence of minimal models (Mori's program)}: Its aim is to provide a nice representative, \emph{a minimal model}, in the birational class of a given variety $X$. A minimal model is required to have a nef canonical divisor and hence it is the ``simplest'' one in its birational class.
\item {\bf Abundance conjecture}: Any minimal model has semiample canonical system so that the $m$-th canonical system is base point free for some $m>0$.
\end{itemize}
If for a variety we can find a minimal model such that abundance holds, then we say this variety has a \emph{good minimal model}. Existence of a good minimal model has been established in several cases. Amongst these
  \begin{itemize}
    \item  ${\rm dim}(X)\leq 3$ by S. Mori, Y. Kawamata, and others,
    \item  varieties of general type by \cite{BCHM}, and
    \item  maximal Albanese dimensional varieties by \cite{OF}.
  \end{itemize}
In this paper we prove the existence of good minimal models in the following cases:
  \begin{itemize}
    \item  The general fiber of the Iitaka fibration of a smooth projective variety $X$ has a good minimal model (Theorem \ref{IF}).
    \item  The general fiber of the Albanese morphism of a smooth projective variety $X$ has dimension no more than three (Theorem \ref{Main}).
  \end{itemize}

\smallskip

The original motivation for this paper is to study {\bf Ueno's Conjecture C} (\cite[\S11]{U}).
\begin{conjecture} If $f:X^n\rightarrow Y^m$ is an algebraic fiber space of smooth projective varieties with general fiber $F$, then we have
    \begin{itemize}
       \item $C_{n,m}:$ $\kappa(X)\geq\kappa(F)+\kappa(Y)$, and
       \item $C_{n,m}^+:$ $\kappa(X)\geq\kappa(F)+{\rm Max}\{{\rm Var}(f),\kappa(Y)\}$ if $\kappa(Y)\geq0$, where ${\rm Var}(f)$ is the variation of $f$ $(\mbox{cf. \cite[\S6 and \S7]{M}})$.
    \end{itemize}
\end{conjecture}
\noindent Conjecture C has also been established in many cases. For example,
   \begin{itemize}
       \item $C^+_{n,m}$ holds if the general fiber $F$ of $f$ has a good minimal model by \cite{Ka1}, and
       \item $C_{n,m}$ holds if the general fiber $F$ of $f$ is of maximal Albanese dimension by \cite{OFa}.
   \end{itemize}
(The reader can find a more complete list of the known results in \cite[\S6 and \S7]{M} which we do not repeat here.) A related conjecture, {\bf Viehweg's Question Q(f)} (cf. \cite[\S7]{M}) asks : Let $f:X\rightarrow Y$ be an algebraic fiber space with ${\rm Var}(f)=\dim(Y)$, then is $f_*(\omega^k_{X/Y})$ big for some positive integer $k$? It is known that a positive answer to $Q(f)$ implies $C^+_{n,m}$. Kawamata proved in \cite{Ka1} that $Q(f)$ holds when the general fiber $F$ has a good minimal model. A question of Mori in \cite[Remark 7.7]{M} then asks if $Q(f)$ holds by assuming that the general fiber of the Iitaka fibration of $F$ has a good minimal model. Hence a corollary of the above mentioned results gives a positive answer to Mori's question:
\begin{cor} Let $f:X\rightarrow Y$ be an algebraic fiber space of normal projective varieties with general fiber $F$. Suppose that the general fiber of the Iitaka fibration of $F$ has a good minimal model. Then Ueno's conjecture C holds on $f$.
\end{cor}

\smallskip

This paper is organized as follow: We recall some definitions in section 1. In section 2 we establish the necessary ingredients for constructing good minimal models. In section 3 we prove a nonvanishing theorem by using generic vanishing results. Section 4 is the heart of this paper where we construct our good minimal models in Theorems \ref{K=0}, \ref{IF}, and \ref{Main}.

\begin{remark} After the completion of this paper, we were informed that Professor Yum-Tong Siu has announced a proof of the abundance conjecture in \cite{S} which in particular would imply many of the results in this paper.
\end{remark}

\bigskip

{\small\noindent {\bf Acknowledgment.} I am grateful to my advisor, Professor Christopher Hacon, for many useful suggestions, discussions, and his generosity.

\section{Preliminaries}

We work over the complex number field $\CC$. We refer the readers to \cite{KM} and \cite{BCHM} for the standard terminology on singularities and the minimal model program, and to \cite{L2} and \cite{BCHM} for definition of a multiplier ideal sheaf and the related asymptotic constructions.

\smallskip

In this paper, \emph{a pair $(X,\Delta)$ over $U$} consists of a $\QQ$-factorial normal projective variety $X$ with an effective $\RR$-Weil divisor $\Delta$ such that $K_X+\Delta$ is $\RR$-Cartier and a projective morphism $X\rightarrow U$ to a quasi-projective variety $U$. We recall the definition of a minimal model here.
\begin{definition}\label{M} For a log canonical pair $(X,\Delta)$ over $U$, we say that a birational map $\phi:(X,\Delta)\dashrightarrow (X',\Delta'=\phi_*\Delta)$ over $U$ is a \emph{minimal model} if
  \begin{enumerate}
     \item $X'$ is normal and $\QQ$-factorial,
     \item $\phi$ extracts no divisors,
     \item $K_{X'}+\Delta'$ is nef over $U$, and
     \item $\phi$ is $(K_X+\Delta)$-negative, i.e. $a(F,X,\Delta)<a(F,X',\Delta')$ for each $\phi$-exceptional divisor $F$.
  \end{enumerate}
Moreover, we say that \emph{abundance} holds on $(X',\Delta')$ if $K_{X'}+\Delta'$ is semiample over $U$, i.e. $K_{X'}+\Delta'$ is an $\RR$-linear sum of $\QQ$-Cartier semiample over $U$ divisors. A \emph{good minimal model} of a pair $(X,\Delta)$ over $U$ is a minimal model such that abundance holds.
\end{definition}
\begin{remark} A minimal model in this paper is a log terminal model as defined in \cite{BCHM}.
\end{remark}
\begin{remark} Let $X\rightarrow U$ and $Y\rightarrow U$ be two projective morphisms of normal quasi-projective varieties. Let $\phi:X\dashrightarrow Y$ be a birational contraction over $U$. Let $D$ and $D'$ be $\RR$-Cartier divisors such that $D'=\phi_*D$ is nef over $U$. Then $\phi$ is $D$-negative if given a common resolution $p:W\rightarrow X$ and $q:W\rightarrow Y$, we may write
\begin{align*} p^*D=q^*D'+E,
\end{align*}
where $p_*E\geq 0$ and the support of $p_*E$ contains the union of all $\phi$-exceptional divisors (cf. \cite[Lemma 3.6.3]{BCHM}).
\end{remark}
\smallskip
A proper morphism $f:X\rightarrow Y$ of normal varieties is \emph{an algebraic fiber space} if it is surjective with connected fibers. For an effective divisor $\Gamma$ on $X$, we write $\Gamma=\Gamma_{\rm hor}+\Gamma_{\rm ver}$ where $\Gamma_{\rm hor}$ and $\Gamma_{\rm ver}$ are effective without common components such that $\Gamma_{\rm hor}$ dominates $Y$ and codim(Supp$(f(\Gamma_{\rm ver})))\geq 1$ on $Y$ respectively.
\smallskip

For general results on Fourier-Mukai transforms, we refer to \cite{Mu}. We recall the definition of certain cohomological support loci which will be used in the proof of Theorem \ref{NONV}.
\begin{definition} Let $\mathcal{F}$ be a coherent sheaf on an abelian variety $A$. Then we define for each $i=0,...,\dim(A)$ the subset
  \begin{align*} V^i(\mathcal{F}):=\{P\in\hat{A}|\ h^i(\mathcal{F}\otimes P)>0\}.
  \end{align*}
\end{definition}
\noindent These subsets are studied in \cite{GL1}, \cite{GL2}, and \cite{Hac}.

\begin{definition} Let $L$ be a line bundle on a smooth projective variety $X$. For each non-negative integer $m$, we define
  \begin{align*} V_m(L):=\{P\in {\rm Pic}^0(X)|\ h^0(X,L^{\otimes m}\otimes P)>0\}.
  \end{align*}
\end{definition}
\noindent These subsets are studied in \cite{CH}.

\section{Preparation}

\subsection{Good Minimal Models}

\begin{lem}\label{MM} Let $(X_i,\Delta_i)$, $i=1,2$, be two klt pairs over $U$ and $\alpha:(X_1,\Delta_1)\dashrightarrow (X_2,\Delta_2)$ be a birational map over $U$ with $\alpha_*\Delta_1=\Delta_2$. Suppose that $\alpha$ is $(K_{X_1}+\Delta_1)$-negative and extracts no divisors, then $(X_1,\Delta_1)$ has a good minimal model over $U$ if $(X_2,\Delta_2)$ does.
\end{lem}
\begin{proof} This is \cite[Lemma 3.6.9]{BCHM}.
\end{proof}

\begin{lem}\label{TM} Let $(X,\Delta)$ be a terminal pair over $U$. Then for any resolution $\mu:(X',\Delta')\rightarrow (X,\Delta)$ with $\Delta':=\mu^{-1}_*\Delta$, a good minimal model of $(X',\Delta')$ is also a good minimal model of $(X,\Delta)$.
\end{lem}
\begin{proof} Note that if we write $K_{X'}+\Delta'=\mu^*(K_X+\Delta)+E$, then $E$ is effective and its support equals to the set of all $\mu$-exceptional divisors. Hence the same argument as in \cite[Lemma 3.6.10]{BCHM} applies (without adding extra $\mu$-exceptional divisors).
\end{proof}

\begin{thm}\label{FLOP} Let $\phi_i:(X,\Delta)\dashrightarrow (X_i,\Delta_i)$, i=1,2, be two minimal models of a klt pair $(X,\Delta)$ over $U$ with  $\Delta_i=(\phi_i)_{*}\Delta$. Then the natural birational map $\psi:(X_1,\Delta_1)\dashrightarrow (X_2,\Delta_2)$ over $U$ can be decomposed into a sequence of $(K_{X_1}+\Delta_1)$-flops over $U$.
\end{thm}
\begin{proof} By \cite[Theorem 3.52]{KM}, $(X_i,\Delta_i)$ are isomorphic in codimension one, and hence the argument in \cite{Ka2} applies.
\end{proof}

\begin{prop}\label{GMM} Let $(X,\Delta)$ be a klt pair over $U$. If $(X,\Delta)$ has a good minimal model over $U$, then any other minimal model of $(X,\Delta)$ over $U$ is also good.
\end{prop}
\begin{proof} Suppose $(X_g,\Delta_g)$ is a good minimal model of $(X,\Delta)$ over $U$ and $(\tilde{X},\tilde{\Delta})$ is another minimal model of $(X,\Delta)$ over U. From Theorem \ref{FLOP}, the birational map $\alpha:(X_g,\Delta_g)\dashrightarrow (\tilde{X},\tilde{\Delta})$ over $U$ may be decomposed into a sequence of flops over $U$. If an intermediate step is $X_i\dashrightarrow X_{i+1}$ with $X_i$ a good minimal model of $(X,\Delta)$ over $U$, then the morphism $X_i\rightarrow Z:={\bf Proj}_U(K_{X_i}+\Delta_i)$ factors through the contraction morphism $g_i:X_i\rightarrow Z_i$ by $\psi:Z_i\rightarrow Z$. Hence for the corresponding flop $g_{i+1}:X_{i+1}\rightarrow Z_i$, there exists a divisor $H$ on $Z$ ample over $U$ such that $K_{X_{i+1}}+\Delta_{i+1}=g_{i+1}^{*}\psi^{*}(H)$. In particular, $K_{X_{i+1}}+\Delta_{i+1}$ is semiample over $U$ and $X_{i+1}$ is also a good minimal model of $(X,\Delta)$ over $U$.
\end{proof}

\begin{prop}\label{TER} If a klt pair $(X,\Delta)$ over $U$ has a good minimal model over $U$, then any $(K_X+\Delta)$ minimal model program with scaling of an ample divisor $A$ over $U$ terminates.
\end{prop}
\begin{proof} Let $\phi:(X,\Delta)\dashrightarrow (X_g,\Delta_g)$ with $\Delta_g=\phi_*\Delta$ be a good minimal model of $(X,\Delta)$ over $U$ and $f:X_g\rightarrow Z={\bf Proj}_U(K_{X_g}+\Delta_g)$ the corresponding morphism over $U$. Note that $\phi$ contracts exactly the divisorial part of ${\bf B}(K_X+\Delta/U)$ (cf. \cite[Lemma 3.6.3]{BCHM}).

\smallskip

Pick $t_0>0$ such that $(X_g,\Delta_g+t_0A_g)$ with $A_g=\phi_*A$ is klt and an ample divisor $H$ on $X_g$. By \cite{BCHM}, the outcome of running a $(K_{X_g}+\Delta_g+t_0A_g)$-minimal model program with scaling of $H$ over $Z$ exists and is a minimal model $\psi:X_g\dashrightarrow X'$ of $(X_g,\Delta_g+t_0A_g)$ over $Z$. As $K_{X_g}+\Delta_g\equiv_Z0$, we have $K_{X'}+\Delta'\equiv_Z0$ where $\Delta'=\psi_*\Delta_g$. Hence those curves contracted in each step of this minimal model program over $Z$ have trivial intersection with $K_{X_g}+\Delta_g$ and negative intersection with $A_g$. In particular, this shows that $X'$ is a minimal model of $(X_g,\Delta_g+tA_g)$ over $Z$ for all $t\in(0,t_0]$. Since $\Delta'+t_0A'$ with $A'=\psi_*A_g$ is big over $U$, there exists only finitely many $(K_{X'}+\Delta'+t_0A')$-negative extremal rays in $\overline{\rm NE}(X'/U)$ by \cite[Corollary 3.8.2]{BCHM}. Hence by considering smaller $t_0>0$, we can assume that $X'$ is a minimal model of $(X_g,\Delta_g+tA_g)$ over $U$ for all $t\in(0,t_0]$. As a map being negative is an open condition, we may choose $t_0>0$ sufficiently small such that $\psi\circ\phi$ is $(K_X+\Delta+tA)$-negative for all $t\in(0,t_0]$, and hence $X'$ is a minimal model of $(X,\Delta+tA)$ over $U$ for all $t\in(0,t_0]$. This implies that $\psi\circ\phi$ contracts exactly the divisorial part of $\mathbf{B}(K_X+\Delta+t_0A/U)$ which is contained in $\mathbf{B}(K_X+\Delta/U)$ and is contracted by $\phi$. Hence $\psi$ contracts no divisors, and in particular $\psi\circ\phi$ is $(K_X+\Delta+tA)$-negative for all $t\in[0,t_0]$. This implies that $X'$ is a minimal model of $(X,\Delta+tA)$ over $U$ for all $t\in[0,t_0]$. Note that then $\mathbf{B}(K_X+\Delta+tA/U)$ has the same divisorial components for all $t\in [0,t_0]$.

\smallskip

Now choose $0<t_1<t_0$ such that $(X,\Delta+t_1A)$ is klt and run a minimal model program of $(X,\Delta+t_1A)$ with scaling of $A$ over $U$. By \cite{BCHM}, the outcome $\phi:X\dashrightarrow \tilde{X}$ exists and is a minimal model of $(X,\Delta+t_1A)$ over $U$. Since being $(K_X+\Delta+tA)$-negative is an open condition and $K_{\tilde{X}}+\tilde{\Delta}+t\tilde{A}:=\phi_*(K_X+\Delta+tA)$ is nef over $U$ for $t\in[t_1,t_0]$, by picking $t_0>0$ smaller if necessary we can assume that $\tilde{X}$ is a minimal model of $(X,\Delta+tA)$ over $U$ for all $t\in[t_1,t_0]$. Since $\mathbf{B}(K_X+\Delta+tA/U)$ has the same divisorial components for all $t\in [0,t_0]$, $X'$ and $\tilde{X}$ are isomorphic in codimension one. {\bf For each $t\in[t_1,t_0]$}, by Theorem \ref{FLOP} we may decompose the birational map $X'\dashrightarrow \tilde{X}$ over $U$ into \emph{possibly different} sequences $S_t$ of $(K_{X'}+\Delta'+tA')$-flops over $U$ as $X'$ and $\tilde{X}$ are both minimal models of $(X,\Delta+tA)$ over $U$. Since $\Delta'+tA'$ is big over $U$ for any $t\in[t_1,t_0]$ and each outcome of a $(K_{X'}+\Delta'+tA')$-flop over $U$ is also a minimal model of $(X,\Delta+tA)$ over $U$, by finiteness of models in \cite{BCHM} we can only have finitely many $(K_{X'}+\Delta'+tA')$-flop over $U$ as $t$ ranges in $[t_1,t_0]$. In particular, we can find an uncountable subset $T_1\subseteq [t_1,t_0]$ such that for all $t\in T_1$, the first $(K_{X'}+\Delta'+tA')$-flops over $U$ of the corresponding sequences $S_t$'s are all the same. Note that those curves contracted by this flop then have trivial intersection with $A'$ and hence this flop is a $(K_{X'}+\Delta')$-flop over $U$. As each sequence $S_t$ is finite, inductively we can find a $t^*\in[t_1,t_0]$ such that all the steps of the sequence $S_{t^*}$ connecting $X'$ and $\tilde{X}$ are $(K_{X'}+\Delta')$-flops over $U$. Since $X'$ is a minimal model of $(X,\Delta)$ over $U$, we then also have that $\tilde{X}$ is a minimal model of $(X,\Delta)$ over $U$. In particular, this shows that the minimal model program of $(X,\Delta)$ with scaling of $A$ over $U$ terminates.
\end{proof}

\begin{cor}\label{AB} Let $(X,\Delta)$ be a klt pair over $U$. Suppose that $(X,\Delta)$ has a good minimal model over $U$, then there exists a $t_0>0$ such that: if $\tilde{X}$ is a minimal model of $(X,\Delta+tA)$ over $U$ for all $t\in[\alpha,\beta]$ for some $0\leq\alpha<\beta\leq t_0$, then $\tilde{X}$ is a minimal model of $(X,\Delta+tA)$ over $U$ for all $t\in[0,t_0]$. In particular, the set of all such minimal models $\tilde{X}$ is finite.
\end{cor}
\begin{proof} By Proposition \ref{TER}, there exists a $t_0>0$ and a birational map $X\dashrightarrow X'$ over $U$ such that $X'$ is a minimal model of $(X,\Delta+tA)$ over $U$ for all $t\in[0,t_0]$. By the proof of Proposition \ref{TER}, there is a finite sequence of $(K_{X'}+\Delta')$-flops over $U$ connecting $X'\dashrightarrow \tilde{X}$ which are also $A'$-trivial and hence $(K_{X'}+\Delta'+tA')$-flops over $U$ for all $t\in[0,t_0]$, where $\Delta'$ and $A'$ are the proper transforms of $\Delta$ and $A$ on $X'$. Therefore the corollary follows. Note that $X'$ and the varieties given by $(K_{X'}+\Delta')$-flops over $U$ appearing in the proof are all minimal models of the big pair $(X,\Delta+t_0A)$ over $U$ and hence by \cite{BCHM} there can only be finitely many of these.
\end{proof}

\begin{prop}\label{AFS} Let $f:X\rightarrow Y$ be an algebraic fiber space of normal quasi-projective varieties such that $X$ is $\QQ$-factorial with klt singularities and projective over $Y$. Suppose that the general fiber $F$ of $f$ has a good minimal model, then $X$ is birational to some $X'$ over $Y$ such that the general fiber of $f':X'\rightarrow Y$ is a good minimal model.
\end{prop}

\begin{proof} Pick an ample divisor $H$ on $X$ and run a minimal model program of $X$ with scaling of $H$ over $Y$. Suppose that ${\rm cont}_R:X\rightarrow W$ is the contraction morphism corresponding to an extremal ray $R\in\overline{\rm NE}(X/Y)$. If $R$ doesn't give an extremal contraction of $F$, then ${\rm cont}_R|_F={\rm id}_F$. Otherwise it's easy to see that ${\rm cont}_R$ and ${\rm cont}_R|_F$ must be of the same type (divisorial or small). Suppose that we have a sequence of infinitely many flips whcih are nontrivial on the general fiber $F$ with $t_i>t_{i+1}>0$ such that $K_{F_i}+tH_i|_{F_i}$ is nef for all $t\in[t_{i+1},t_i]$. Since $F$ has a good minimal model, by Corollary \ref{AB} the set of such $F_i$'s is finite (modulo isomorphisms) and each $F_i$ is a good minimal model of $F$. Then we get a contradiction by the same argument as in the last step of the proof of \cite[Lemma 4.2]{BCHM}. Hence after finitely many steps, we may assume that all flips are trivial on the general fiber, and so we get an algebraic fiber space $f':X'\rightarrow Y$ such that the general fiber is a good minimal model.
\end{proof}

\subsection{Degenerate Divisors}

This part concerns the negativity property of a ``degenerate'' divisor. The following definition is taken from \cite{Tak}.

\begin{definition} Let $f:X\rightarrow Y$ be a proper surjective morphism of normal varieties and $D\in {\rm WDiv}_{\RR}(X)$ be an effective Weil divisor. Then
  \begin{itemize}
    \item $D$ is $f$-exceptional if ${\rm codim(Supp}(f(D)))\geq 2$.
    \item $D$ is of insufficient fiber type if ${\rm codim(Supp}(f(D)))=1$ and there exists a prime divisor $\Gamma\nsubseteq{\rm Supp}(D)$ such that $f(\Gamma)\subseteq{\rm Supp}(f(D))$ has codimension one in $Y$.
  \end{itemize}
In either of the above cases, we say that $D$ is \emph{degenerate}. In particular, a degenerate divisor is always assumed to be effective.
\end{definition}

\begin{lem}\label{CONT} Let $f:X\rightarrow Y$ be an algebraic fiber space of normal projective varieties such that  $X$ is $\QQ$-factorial. Then for a degenerate Weil divisor $D$ on $X$, we can always find a component $F\subseteq {\rm Supp}(D)$ which is covered by curves contracted by $f$ and intersecting $D$ negatively. In particular, we have $F\subseteq \mathbf{B}_{-}(D/Y)$, the diminished base locus of $D$ over $Y$.
\end{lem}

\begin{proof} Write $D=\Sigma r_iD_i$ with $r_i>0$ and $D_i\in {\rm Div}(X)$ prime.

\smallskip

Case 1: Suppose $D$ is $f$-exceptional, and hence $\dim Y\geq 2$. Cutting by general hyperplanes, we reduce to a birational morphism of surfaces with $E=\Sigma r_j\tilde{E}_j$, where $\tilde{E}_j=D_j\cap H_1\cap ...\cap H_{n}$ may be nonreduced and reducible and $E=D\cap H_1\cap ...\cap H_{n}$. Note that we may assume $P:=f(E)$ is a point. By Hodge index theorem (cf. \cite[Corollary 2.7]{LB}), the intersection matrix of irreducible components of $f^{-1}(P)$ is negative-definite. Then $(E)^2<0$, and hence $(\tilde{E}_j.D)=(\tilde{E}_j.E)<0$ for some $j$. In particular, $(\tilde{E}_j.D_j)<0$ and $D_j$ is covered by curves intersecting $D$ negatively.

\smallskip

Case 2: Suppose $D$ is of insufficient fiber type. Cutting by general hyperplanes, we reduce to a morphism from a surface to a curve with $E=\Sigma r_j\tilde{E}_j$ supported on fibers, where $E_j=D_j\cap H_1\cap...\cap H_{n}$ may be nonreduced and reducible and $E=D\cap H_1\cap ...\cap H_{n}$. Then by \cite[Corollary 2.6]{LB}, we have $(E)^2\leq 0$. But ${\rm Supp}(E)$ can not be the whole fiber, hence we can find $\Gamma$ an effective divisor having no common components with $E$ such that ${\rm Supp}(E+\Gamma)=f^{-1}(f(E))$. For $F:=f^*(f_*(E))$, then we can find $a$ and $b$ two positive real numbers such that $aF\leq E+\Gamma\leq bF$. If $(E)^2=0$, then $E$ is nef and hence $E.F=0$ implies $E.(E+\Gamma)=0$. But we have $E.\Gamma>0$ which implies $(E)^2<0$, a contradiction. Hence $(E)^2<0$ and the same argument as in case 1 applies.

\smallskip

To prove that $D_j\subseteq \mathbf{B}_{-}(D/Y)$, we pick an ample divisor $A$ on $X$ and $\epsilon>0$ a small rational number such that $\tilde{E}_j.(D+\epsilon A)<0$. Note that we then also have  $\tilde{E}_j.(D+\epsilon A+f^*R)<0$ for any $\RR$-Cartier divisor $R$ on $Y$. In particular, this shows that $\tilde{E}_j\subseteq \mathbf{B}(D+\epsilon A/Y)$. As $\tilde{E}_j$ passes through a general point of $D_j$, we have $D_j\subseteq \mathbf{B}(D+\epsilon A/Y)\subseteq\mathbf{B}_{-}(D/Y)$.
\end{proof}

\section{Nonvanishing Theorems}

\begin{thm}\label{NONV} Let $X$ be a smooth projective irregular variety with $\alpha:=alb_X:X\rightarrow A:=Alb(X)$ the Albanese morphism and $\alpha ':X\rightarrow Y$ with general fiber $F$ be the Stein factorization of $\alpha:X\rightarrow \alpha(X)\subseteq A$. Suppose $\kappa(F)\geq0$, then $\kappa(X)\geq0$.
\end{thm}

\begin{lem} Assumptions as in Theorem \ref{NONV}, then $K_X$ is pseudo-effective.
\end{lem}
\begin{proof} We have $\alpha_*'\omega_{X/Y}^N\neq0$ and is weakly positive by \cite{V1}. Hence for any $\epsilon>0$ and $H$ ample on $Y$, $\alpha_*'\omega_{X/Y}^N\otimes(\epsilon H)$ is big. As $Y$ is finite over $\alpha(X)$, a subvariety in $A$, we have $\kappa(Y)\geq0$ and hence $\alpha_*'\omega_{X}^N\otimes(\epsilon H)$ is also big. In particular $\kappa(K_X+\frac{\epsilon}{N}(\alpha')^*H)\geq0$ for any $\epsilon>0$, and hence $K_X$ is pseudo-effective.
\end{proof}

\begin{lem}\label{H} Let X be a smooth projective variety. Suppose $\{D_k\}$ is a collection of effective $\QQ$-divisors with $k\in\mathbf{N}$ such that the corresponding multiplier ideal sheaves $\mathcal{J}_{k}:=\mathcal{J}(D_{k})$ satisfy $\mathcal{J}_{k}\subseteq\mathcal{J}_{k'}$ whenever $k\leq k'$. If there exists a line bundle $L$ such that $L-D_{k}$ is nef and big for all $k>0$, then $\bigcap_{i>0}\mathcal{J}_{i}=\mathcal{J}_{k}$ for $k$ sufficiently large.
\end{lem}
\begin{proof} The proof is taken from \cite[Proposition 5.1]{Hac}. We reproduce the proof here for the convenience of the reader. Take a sufficiently ample divisor $H$ on $X$ and consider the line bundle $M=L+(n+1)H$ for $n={\rm dim}(X)$, then
   \begin{align*} M-D_{k}-(iH)\equiv L-D_{k}+(n-i+1)H
   \end{align*}
is nef and big for all $k>0$ and $1\leq i\leq n$. Hence we have $ H^i(X,\mathcal{O}_X(K_X+M-iH)\otimes\mathcal{J}_{k})=0$ for all $i>0$ by Nadel vanishing, and then $\mathcal{O}_X(K_X+M)\otimes\mathcal{J}_{k}$ is generated by global sections by Mumford regularity. In particular, if $\mathcal{J}_{k}\neq\mathcal{J}_{k'}$ for $k\leq k'$, then we get a strict inclusion $H^0(X,\mathcal{O}_X(K_X+M)\otimes\mathcal{J}_{k})\subseteq H^0(X,\mathcal{O}_X(K_X+M)\otimes\mathcal{J}_{k'})$ of $\CC$ vector spaces. But this can not happen for infinitely many times, hence the lemma follows.
\end{proof}

\begin{lem}\label{G} The same setting as in Theorem \ref{NONV}. Then for $H$ an ample divisor on $A$ and a non-negative integer $m$, $\mathcal{J}(\|mK_X+\epsilon\alpha^{*}H\|)$ is independent of $\epsilon\in\QQ$ for any $\epsilon>0$ sufficiently small. Hence we can define the sheaf
   \begin{align*} \mathcal{F}_m:=\alpha_*(\omega_X^m\otimes\mathcal{J}(\|(m-1)K_X+\epsilon\alpha^{*}H\|))
   \end{align*}
on $A$ for $\epsilon>0$ a sufficiently small rational number. Then for $L$ any sufficiently ample line bundle on the dual abelian variety $\hat{A}$ with $\hat{L}$ the Fourier-Mukai transform of $L$ on $A$, we have $H^{i}(A,\mathcal{F}_m\otimes\hat{L}^{\vee})=0$ for all $i>0$. From \cite[Corollary 3.2]{Hac}, we then have for any non-negative integer $m$ the inclusions:
   \begin{align*} V^0(\mathcal{F}_m)\supseteq V^1(\mathcal{F}_m)\supseteq...\supseteq V^n(\mathcal{F}_m).
   \end{align*}
In particular, $V^{0}(\mathcal{F}_m)=\phi$ implies $\mathcal{F}_m$=0.
\end{lem}
\begin{proof} The first statement follows from Lemma \ref{H} by taking $L$ to be $mK_X+\alpha^*H$ on $X$. The vanishing of cohomologies follows from \cite[Theorem 4.1]{Hac} with a slight modification and hence we reproduce the argument here. Consider the isogeny $\phi_L:\hat{A}\rightarrow A$ defined by $L$, $\hat{\alpha}:\hat{X}\rightarrow \hat{A}$, and $f:\hat{X}=X\times_A\hat{A}\rightarrow X$. Then as $\phi^{*}_L\hat{L}^{\vee}=\oplus_{h^0(L)}L$, we have
   \begin{align*} H^i(A,\mathcal{F}_m\otimes\hat{L}^{\vee})
                                       &\subseteq H^i(A,\mathcal{F}_m\otimes\hat{L}^{\vee}\otimes{\phi_L}_{*}\mathcal{O}_{\hat{A}})\\
                                       &=H^i(\hat{A},\phi_L^{*}\mathcal{F}_m\otimes\phi^{*}_L\hat{L}^{\vee})\\
                                       &=\oplus H^i(\hat{A},\hat{\alpha}_*f^*(\omega_X^m\otimes\mathcal{J}(\|(m-1)K_X+\epsilon\alpha^{*}H\|))\otimes L) \\
                                       &=\oplus H^i(\hat{A},\hat{\alpha}_*(\omega^m_{\hat{X}}\otimes\mathcal{J}(\|(m-1)K_{\hat{X}}+\epsilon\hat{\alpha}^{*}\phi_L^*H\|))\otimes L),
   \end{align*}
where the last equality is the  ${\rm \acute{e}}$tale base change of multiplier ideal sheaves in \cite[Theorem 11.2.16]{L2}. For $i>0$, the cohomological groups above vanish by Nadel vanishing on $\hat{X}$, or by Kawamata-Viehweg vanishing theorem on a log resolution $\pi:Y\rightarrow\hat{X}$. The final statement follows from \cite[Theorem 2.2]{Mu}.
\end{proof}

\begin{proof}\emph{(of Theorem \ref{NONV})} For general point $z\in Y$ and $m$ sufficiently divisible, we have for the sheaves defined by $\mathcal{F}_m':=\alpha_*'(\omega_X^m\otimes\mathcal{J}(\|(m-1)K_X+\epsilon\alpha^{*}H\|))$ on $Y$:
   \begin{align*} (\mathcal{F}_m')_z &= H^0(F,\omega_F^m\otimes\mathcal{J}(\|(m-1)K_X+\epsilon\alpha^{*}H\|)|_F) \\
                                    &\supseteq H^0(F,\omega_F^m\otimes\mathcal{J}(\|(m-1)K_X+\epsilon\alpha^{*}H\|_F)) \\
                                    &= H^0(F,\omega_F^m\otimes\mathcal{J}(\|(m-1)K_F\|)) \\
                                    &\supseteq H^0(F,\omega_F^m\otimes\mathcal{J}(\|mK_F\|)) \\
                                    &= H^0(F,\omega_F^m)>0.
   \end{align*}
The first inclusion follows from the properties of the restriction of multiplier ideal sheaves in \cite[Theorem 11.2.1]{L2}, the second equality from the explanation of semipositivity in \cite[Proposition 10.2]{K}, and the last inequality from $\kappa(F)\geq0$. Hence $\mathcal{F}_m'$ is non-trivial. In particular, $\mathcal{F}_m$ is also non-trivial for $m$ sufficiently divisible.

\smallskip

For $m$ sufficiently divisible, $\mathcal{F}_m\neq0$ and hence $V^0(\mathcal{F}_m)\neq\phi$ by Lemma \ref{G}. This shows that we can find an element $P\in {\rm Pic}^0(X)$ with $H^0(X,\omega_X^m\otimes P)\neq 0$. Following the argument of \cite[Theorem 3.2]{CH} (cf. Theorem \ref{T}), $V_m(K_X)$ is a union of torsion translates of subvarieties in ${\rm Pic}^0(X)$ for $m\geq 1$ and in particular we can find an element $P'\in {\rm Pic}^0(X)_{\rm tor}$ with $H^0(X,\omega_X^m\otimes P')\neq 0$. Then $H^0(X,\omega_X^{md})\neq 0$ for $d={\rm ord}(P')$ in ${\rm Pic}^0(X)$ and hence $\kappa(X)\geq 0$.
\end{proof}

\begin{thm}\label{T} Let $X$ be a smooth projective variety. Then the cohomological loci
\begin{align*} V_m(K_X):=\{P\in {\rm Pic}^0(X)|\ h^0(X,\omega_X^{\otimes m}\otimes P)>0\}
\end{align*}
for $m$ a positive integer, if non-empty, is a finite union of torsion translates of abelian subvarieties of ${\rm Pic}^0(X)$.
\end{thm}
\begin{proof} If $m=1$, then by a result of Simpson in \cite{Sim} the loci $V_1(K_X)$ is a union of torsion translates of abelian subvarieties of ${\rm Pic}^0(X)$. In general, let $\tilde{P}\in V_m(K_X)$ and write $\tilde{P}=mP$ for some $P\in{\rm Pic}^0(X)$. Let $\mu:X'\rightarrow X$ be a log resolution of $|m(K_X+P)|$, and $D\in\mu^*|m(K_X+P)|$ be a divisor with simple normal crossing support. Consider the line bundle $N:=\mu^*\mathcal{O}_{X}((m-1)(K_X+P))\otimes\mathcal{O}_{X'}(-\lfloor\frac{m-1}{m}D\rfloor)$, then it follows from \cite[Theorem 8.3]{CH2} and \cite{Sim} that the cohomological loci
\begin{align*} V^0(\omega_{X'}\otimes N):=\{R\in{\rm Pic}^0(X')| h^0(\omega_{X'}\otimes R)>0\}
\end{align*}
is a union of torsion translates of abelian subvarieties of ${\rm Pic}^0(X')$. Note that ${\rm Pic}^0(X')\cong{\rm Pic}^0(X)$ as $X$ is smooth, and hence we may identify the elements in these two groups (via pulling back by $\mu$). It is easy to see that $P\in V^0(\omega_{X'}\otimes N)$, and hence there exists an abelian subvariety $T\subseteq{\Pic}^0(X)$ and a torsion element $Q\in{\rm Pic}^0(X)_{\rm tor}$ such that
\begin{align*} P\in T+Q\subseteq V^0(\omega_{X'}\otimes N).
\end{align*}
By pushing forward, it is also easy to see that
\begin{align*} T+Q+(m-1)P\subseteq V_m(K_X).
\end{align*}
Now since $rP\in rT$ for some positive integer $r$ and $rT$ is a group, we have that $r(m-1)P\in rT$ and hence $(m-1)P\in T+Q'$ for some torsion element $Q'\in{\rm Pic}^0(X)_{\rm tor}$. In particular, we have
\begin{align*} \tilde{P}=mP\in T+Q+(m-1)P=T+Q+Q' \subseteq V_m(K_X),
\end{align*}
and hence $V_m(K_X)$ is a union of torsion translates of abelian subvarieties of ${\rm Pic}^0(X)$. 

\smallskip

Let $V$ be an irreducible component of $V_m(K_X)$ and denote ${\rm Pic}^0(X)$ by $A$. Note that for any general point of $V$, there is a torsion translate of an abelian subvariety of $A$  contained in $V$ passing through it. It is well-known that if $V$ is of general type, then there are no nontrivial abelian subvarieties of $A$ contained in $V$ passing through general points of $V$. In this case, a general point of $V$ must be torsion and hence $\dim V$ can only be zero since there are only countably many torsion points in $A$. It follows that $V$ is a torsion point. If $V$ is not of general type, then by \cite[Theorem 10.9]{U} there is an algebraic fiber space $f:V\rightarrow B$ with general fiber $A_1$ induced by $\pi:A\rightarrow A/A_1$, where $A_1$ is an abelian subvariety of $A$ and $B\subseteq A/A_1$ is a subvariety of general type. Since there are also torsion translate of abelian subvarieties of $A/A_1$ contained in $B$ passing through general points of $B$, $B$ is a torsion point and then $V$ is a torsion translate of an abelian subvariety of $A$. Hence we conclude that the algebraic set $V_m(K_X)$, if non-empty, is a \emph{finite} union of torsion translates of abelian subvarieties of ${\rm Pic}^0(X)$. 
\end{proof}

\section{Main Theorems}

We will now establish the existence of good minimal models in several different situations.

\subsection{Kodaira Dimension $\kappa(X)=0$}

\begin{lem}\label{Exc} Let $X$ be a smooth projective variety with $\kappa(X)=0$ and $\alpha:X\rightarrow A:=Alb(X)$ be the Albanese morphism. Suppose that $|mK_X|\neq\phi$ for some $m>0$ and $F$ is the unique effective divisor in $|mK_X|$, then ${\rm Supp}(F)$ contains all $\alpha$-exceptional divisors.
\end{lem}
\begin{proof} Suppose that we have $P_1(X)=P_2(X)=1$, where $P_n(X):=h^0(X,\omega_X^n)$ is the $n$-th plurigenus. Then by \cite[Proposition 2.1]{EL}, the origin of ${\rm Pic}^0(X)$ is an isolated point of the cohomological support loci
\begin{align*} V_0(X):=\{y\in {\rm Pic}^0(X)|h^0(X,\omega_X\otimes P_y)\neq 0\}.
\end{align*}
By \cite[Corollary 1.5]{EL}, this implies that for every non-zero $\eta\in H^0(X,\Omega_X^1)$, the map
\begin{align*} \phi_{\eta}:H^0(X,\Omega_X^{n-1})\stackrel{\wedge\eta}{\rightarrow} H^0(X,\Omega_X^n)
\end{align*}
determined by wedging with $\eta$ is surjective. Now for a given $\alpha$-exceptional divisor $E$, the differential $d\alpha:T_eX\rightarrow T_{\alpha(e)}A$ is not of full rank at a general point $e\in E$. In particular, if $\eta_e$ is the non-zero $1$-form given by pulling back a flat $1$-form on $A$ which is in $\ker[(d\alpha)^{\vee}:T^*_{\alpha(e)}A\rightarrow T^*_eX]$, then the surjectivity of $\phi_{\eta_e}$ shows that $F$ must pass through $e$ (cf. \cite[Proposition 2.2]{EL}). Since $e\in E$ is a general point, we have $E\subset F$.

\smallskip

For the general case, let $\mu:X'\rightarrow X$ be a log resolution of $(X,F)$. Then we may write $mK_{X'}=\mu^*(mK_X)+E\sim \mu^*F+E=:F'\geq 0$ where $E$ is effective and consists of $\mu$-exceptional divisors, ${\rm Supp}(F')$ is a simple normal crossing divisor, and $\mu_*F'=F$. We take a cyclic cover $f:Y\rightarrow X'$ defined by $F'\in |mK_{X'}|$ followed by a resolution $d:Y'\rightarrow Y$. It is well-known that $Y$ is normal with only quotient singularities and for $f':=f\circ d$ we have
\begin{align*} f'_*\mathcal{O}_{Y'}=\bigoplus^{m-1}_{i=0}(\omega^i_{X'}(-\lfloor\frac{i}{m}F'\rfloor))^\vee \end{align*} 
and
\begin{align*} f'_*\omega_{Y'}=\omega_{X'}\otimes(\bigoplus^{m-1}_{i=0}\omega^i_{X'}(-\lfloor\frac{i}{m}F'\rfloor)).\end{align*}
An easy computation then shows that $\kappa(Y')=0$ and $P_1(Y')=P_2(Y')=1$. As the generic vanishing result still holds for the induced map $\alpha':Y'\rightarrow A$ (cf. \cite[Remark 1.6]{EL}), the argument above then shows that for $m>0$ sufficiently divisible and the unique effective divisor $\Gamma\in|mK_{Y'}|$, ${\rm Supp}(\Gamma)$ contains all $\alpha'$-exceptional divisors. It is then easy to see that ${\rm Supp}(F')={\rm Supp}(f'(\Gamma))$ and hence the lemma now follows by pushing forward $\Gamma$ to $X$.
\end{proof}

\begin{thm}\label{K=0} Let $X$ be a normal projective $\QQ$-factorial variety with at most terminal singularities and $\kappa(X)=0$. Suppose the general fiber $F$ of the Albanese morphism has a good minimal model, then $X$ has a good minimal model.
\end{thm}

\begin{proof} By Lemma \ref{TM}, we may assume that $X$ is smooth. By \cite{Ka}, the Albanese map $\alpha:=alb_X:X\rightarrow A:=Alb(X)$ is an algebraic fiber space. Moreover we have $\kappa(F)=0$ as $C_{n,m}$ holds by \cite{Ka1}.

By Proposition \ref{AFS}, after running a minimal model program of $X$ with scaling of an ample divisor over $A$, we have a birational map $X\dashrightarrow X'$ over $A$ such that the general fiber $F'$ of $\alpha':X'\rightarrow A$ is a good minimal model. Moreover we may assume that ${\bf B}_{-}(K_{X'}/A)$ contains no divisorial components. Note that then $\kappa(F')=0$ implies $K_{F'}\sim_\QQ 0$. For $K_{X'}\sim_\QQ\Gamma$ with $\Gamma$ effective, we write $\Gamma=\Gamma_{\rm hor}+\Gamma_{\rm ver}$. Then as $\Gamma_{\rm hor}|_{F'}=\Gamma|_{F'}\sim_\QQ K_{X'}|_{F'}\sim K_{F'}\sim_\QQ 0$, we have $\Gamma_{\rm hor}=0$. Suppose there exists an effective divisor $E\leq\Gamma$ with $P:=\alpha'_*(E)_{\rm red}$ a codimension one point and $E$ contains all divisors on $X'$ dominating $P$. Then we have ${\rm Supp}((\alpha')^{-1}(P))\subseteq {\rm Supp}(\Gamma)$ (note that $(\alpha')^{-1}(P)$ may have some exceptional divisorial components which are automatically contained in ${\rm Supp}(\Gamma)$ by Lemma \ref{Exc}). This implies that
  \begin{align*} 0=\kappa(X)=\kappa(X')\geq\kappa(\mathcal{O}_{X'}(\Gamma))\geq\kappa(\mathcal{O}_A(P))>0,
  \end{align*}
a contradiction. Hence $\Gamma$ is of insufficient fiber type. By Lemma \ref{CONT}, we can find a component $D$ of $\Gamma$ such that $D\subseteq\mathbf{B}_{-}(K_{X'}/A)$. But this is impossible as ${\bf B}_{-}(K_{X'}/A)$ contains no divisorial components. Hence $K_{X'}\sim_\QQ0$ and $X'$ is a good minimal model of $X$.
\end{proof}

\begin{cor}\label{K=0'} Let $X$ be a projective variety with terminal singularities and $\kappa(X)=0$. Let $V$ be a smooth projective variety of maximal Albanese dimension and $\alpha:X\rightarrow V$ be an algebraic fiber space. If the general fiber $F$ of $\alpha$ has a good minimal model, then $X$ has a good minimal model.
\end{cor}
\begin{proof} By \cite{Ka1}, we have $\kappa(V)=0$ and hence $V$ is birational to its Albanese variety $A:=A(V)$. We may then replace $V$ by $A$. By Proposition \ref{TM}, we may assume that $X$ is smooth. Then we have $\alpha:X\rightarrow A$ an algebraic fiber space such that the general fiber $F$ has a good minimal model. As noted in the proof of Lemma \ref{Exc}, the argument in Lemma \ref{Exc} and Theorem \ref{K=0} still works in this general case. Hence the corollary follows.
\end{proof}

\subsection{Iitaka Fibration}

\begin{thm}\label{IF} Let $X$ be a $\QQ$-factorial normal projective variety with non-negative Kodaira dimension and at most terminal singularities. Suppose the general fiber $F$ of the Iitaka fibration has a good minimal model, then $X$ has a good minimal model.
\end{thm}

\begin{proof}
The theorem is certainly true for the case $\kappa(X)=0$. For varieties of general type, the theorem follows from \cite{BCHM} and base point free theorem in \cite{KM}. Hence we may assume $0<\kappa(X)<{\rm dim}(X)$.

\smallskip

By \cite{BCHM}, $R(K_X)$ is a finitely generated $\CC$-algebra and hence there is an integer $d$ such that the truncated ring $R^{[d]}(K_X)$ is generated in degree $1$. Take a resolution $\mu:X'\rightarrow X$ of $X$ and $|dK_X|$, then
   \begin{itemize}
     \item $\mu^{*}|mdK_X|=|mM|+mF$  with $|mM|$ base point free and $mF\geq 0$ the fixed divisor for all $m>0$,
     \item $f:=\phi_{|M|}:X'\rightarrow Y$ is birationally equivalent to the Iitaka fibration,
     \item $K_{X'}=\mu^{*}K_X+E$ with $E$ effective and $\mu$-exceptional,
     \item $dK_{X'}\sim M+F+dE$ with $F+dE$ effective and $F+dE\subseteq {\bf B}(K_{X'})$.
   \end{itemize}

\smallskip

Write $\Gamma:=F+dE=\Gamma_{\rm hor}+\Gamma_{\rm ver}$ with respect to $f$. By Proposition \ref{AFS}, after running a minimal model program of $X'$ with scaling of an ample divisor over $Y$, we may assume that the general fiber of $f$ is a good minimal model. As the general fiber $F$ of $f$ has Kodaira dimension zero, we have $\Gamma_{\rm hor}|_F=(M+F+dE)|_F\sim (dK_{X'})|_F\sim dK_F\sim_\QQ 0$ and hence $\Gamma_{\rm hor}=0$. In particular, we may assume $F+dE$ consists of only vertical divisors. Note that the condition $F+dE\subseteq {\bf B}(K_{X'})$ still holds after steps of a minimal model program.

\smallskip

Consider $T$ an effective divisor with ${\rm Supp}(T)\subseteq {\rm Supp}(F+dE)$ and the exact sequences
   \begin{align*} 0\rightarrow f_*\mathcal{O}_{X'}((j-1)T)\rightarrow f_*\mathcal{O}_{X'}(jT)\rightarrow Q_j\rightarrow 0\
   \end{align*}
on $Y$ with $j\geq1$ and $Q_j$ the quotient. After tensoring with $\mathcal{O}_Y(k)$ for $k$ sufficiently large, we have $Q_j(k)$ is generated by global sections and $H^1(Y,f_*\mathcal{O}_{X'}((j-1)T)\otimes\mathcal{O}_Y(k))=0$. As $T\subseteq {\bf B}(K_{X'})$, we have
   \begin{align*}
     H^0(Y,f_*\mathcal{O}_{X'}(jT)\otimes\mathcal{O}_Y(k))=H^0(X',\mathcal{O}_{X'}(kM+jT))=H^0(X',\mathcal{O}_{X'}(kM))=H^0(Y,\mathcal{O}_Y(k))
   \end{align*}
for any $j\geq 0$ as $\mathcal{O}_{X'}(M)=f^*\mathcal{O}_Y(1)$. Hence the exact sequence of cohomological groups shows that $H^0(Y,Q_j(k))=0$ and then $Q_j=0$. In particular, $f_*\mathcal{O}_{X'}(jT)=\mathcal{O}_Y$ for any $j\geq0$. Suppose that $f_*(T)_{\rm red}=P$ is a codimension one point on $Y$ such that ${\rm Supp}(T)$ contains all divisors in $X'$ dominating $P$. Note that we can find a big open subset $U\subseteq Y$ such that the image of the exceptional divisors contained in $f^*(P)$ is disjoint from $U$ as it has codimension greater or equal to two. Then there is a positive integer $j$ such that $f_*\mathcal{O}_{X'}(jT)|_U\supseteq\mathcal{O}_Y(P)|_U$. Since $f_*\mathcal{O}_{X'}(jT)=\mathcal{O}_Y$ and $\mathcal{O}_Y(P)$ are reflexive, we have an inclusion $\mathcal{O}_Y(P)\subseteq\mathcal{O}_Y$ which is impossible. In particular, this shows that $F+dE$ is of insufficient fiber type over $Y$.

\smallskip

By Lemma \ref{CONT}, we can find a component of $F+dE$ which is contained in ${\bf B}_{-}(K_{X'}/Y)$. The same argument as in Theorem \ref{K=0} then says that this is impossible. Then $dK_{\tilde{X}}\sim {M}$ with $\mathcal{O}_{X'}(M)=f^*\mathcal{O}_Y(1)$ is base point free and hence $X'$ is a good minimal model of $X$ by Lemma \ref{TM} (as $\mu$ is a resolution of a terminal variety).
\end{proof}

\subsection{Albanese morphism}

\begin{thm}\label{Main} Let $X$ be a smooth projective variety with Albanese map $\alpha:X\rightarrow A:=Alb(X)$. If the general fiber $F$ of $\alpha$ has dimension no more than three with $\kappa(F)\geq 0$, then X has a good minimal model.
\end{thm}
\begin{proof} By Theorem \ref{NONV}, we have $\kappa(X)\geq 0$. Let $X\dashrightarrow Z$ be the Iitaka fibration and $X'\rightarrow X$ a resolution such that $X'\rightarrow Z$ is a morphism. By \cite[Lemma 2.6]{CH}, the image of the general fiber $X'_z$ of $X'\rightarrow Z$ over a general point $z\in Z$ is a translation of a fixed abelian subvariety $K$ in $A$ and $0\rightarrow K\rightarrow A\rightarrow Alb(Z)\rightarrow 0$ is exact. Consider the Albanese map $\alpha_z:X'_z\rightarrow Alb(X'_z)$ of $X'_z$ which is an algebraic fiber space as $\kappa(X'_z)=0$. As $K$ is an abelian variety, $X'_z\rightarrow K$ factors through $\alpha_z$ by a surjective map $Alb(X'_z)\rightarrow K$. In particular, the general fiber $F_z$ of $\alpha_z:X'_z\rightarrow Alb(X'_z)$ is contained in $F$ and hence has dimension no more than three. Then $X'_z$ has a good minimal model by Theorem \ref{K=0} as $F_z$ does. Since $X'_z$ is the general fiber of the Iitaka fibration of $X$, Theorem \ref{IF} implies $X$ also has a good minimal model.
\end{proof}

\begin{remark} In fact, we have
   \begin{align*} \dim(F_z)\leq\dim(F)-\kappa(X)+q(Z),
   \end{align*}
where $q(Z)$ is the irregularity of $Z$. Hence if $\dim(F)-\kappa(X)+q(Z)\leq 3$, then Theorem \ref{Main} still holds.
\end{remark}


\begin{thebibliography}{ELMNP}

\bibitem[B$\breve{\rm a}$d01]{LB}L.\ B$\breve{\rm a}$descu:\ \emph{Algebraic Surfaces},\ Universitext,\ Spring-Verlag,\ New\ York,\ 2001,\ Translated\ from\ the\ 1981\ Romanian\ original\ by\ Vladimir\ Masek\ and\ revised\ by\ the\ author.MR1805816(2001k:14068)

\bibitem[BCHM06]{BCHM} C. Birkar, P. Cascini, C. Hacon, and J. McKernan: \emph{Existence of minimal models for varieties of log general type}, 2006. arXiv:math.AG/0610203

\bibitem[CH]{CH} J.A. Chen and C. Hacon: \emph{On the Irregularity of Image of Iitaka fibration}. Comm. in Albegra Vol. {\bf 32}, No. 1,pp. 203-125(2004).

\bibitem[CL-H]{CH2} H.\ Clemens\ and\ C.\ Hacon:\ \emph{Deformations of the trivial line bundle and vanishing theorems},\ Amer.\ J.\ Math.\ 2002,\ 111(1),\ 159-175.

\bibitem[EL97]{EL} L.\ Ein\ and\ R.\ Lazarsfeld:\ \emph{Singularities of theta divisors and the birational geometry of irregular varieties},\ J.\ Amer.\ Math.\ Soc.\ 10\ (1997),\ no.\ 1,\ 243--258.

\bibitem[Fuj02a]{OFa} O.\ Fujino:\ \emph{Algebraic fiber spaces whose general fibers are of maximal Albanese dimension}, 2002. arxiv:math.AG/0204262

\bibitem[Fuj02b]{OFb} O.\ Fujino:\ \emph{Remarks on algebraic fiber spaces}, 2002. arxiv:math.AG/0210166

\bibitem[Fuj09]{OF} O.\ Fujino:\ \emph{On maximal Albanese dimensional varieties}, 2009. arxiv:math.AG/0911.2851

\bibitem[GL1]{GL1} M. Green and R. Lazarsfeld: \emph{Deformation theory, generic vanishing theorems, and some conjectures of Enriques, Catanese and Beauville}. Inventiones Math {\bf 90} (1987),\ 389-407.

\bibitem[GL2]{GL2} M. Green and R. Lazarsfeld: \emph{Higher obstructions to deforming cohomology groups of line bundles}, Jour. Amer. Math. Soc. vol. 4, {\bf 1}, 1991, 87-103.

\bibitem[Hac04]{Hac} C. Hacon: \emph{A derived category approach to generic vanishing theorems}. Journ. Reine Angew. Math. 575 (2004).\ arXiv:math.AG/0308195.

\bibitem[HK]{HK} C.\ Hacon\ and\ S.\ Kov$\acute{\rm a}$cs:\ \emph{Recent advances in the classification of algebraic varieties},\ preprint.

\bibitem[Har77]{H} R.\ Hartshorne:\ \emph{Algebraic Geometry},\ Springer-Verlag,\ New\ York,\ 1977,\ Graduate\ Texts\ in\ Mathematics,\ Np.\ 52.MR0463157(57\ \#3116)

\bibitem[Kaw81]{Ka} Y.\ Kawamata:\ \emph{Characterization of abelian varieties},\ Compositio\ Math.\ {\bf 43},\ 1981,\ 253-276.

\bibitem[Kaw85]{Ka1} Y.\ Kawamata:\ \emph{Minimal models and the Kodaira dimension of algebraic fiber spaces}, J.\ Reine\ Angew.\ Math.\ {\bf 363}(1985),\ 1-46.

\bibitem[Kaw07]{Ka2} Y.\ Kawamata:\ \emph{Flops connect minimal model},\ 2007.\ arxiv:math.AG/0704.1013v1

\bibitem[KM98]{KM} J.\ Koll$\acute{\rm a}$r\ and S.\ Mori:\ \emph{Birational geometry of algebraic varieties},\ Cambridge\ University\ Press,\ 1998.

\bibitem[Kol93]{K} J.\ Koll$\acute{\rm a}$r:\ \emph{Shafarevich maps and plurigenera of algebraic varieties},\ Inv.\ Math.\ 113\ (1993),\ 177-215.

\bibitem[Laz04a]{L1} R.\ Lazarsfeld:\ \emph{Positivity in algebraic geometry. I},\ \ Ergebnisse\ der\ Mathematik\ und\ ihrer\ Grenzgebiete\. 3.\ Folge.\ A\ Series\ of\ Modern\ Surveys\ in\ Mathematics\ [Results\ in\ Mathematics\ and\ Related\ Areas.\ 3rd\ Series.\ A\ Series\ of\ Modern\ Surveys\ in\ Mathematics],\ vol.\ 48,\ Springer-Verlag,\ Berlin,\ 2004,\ Classical\ setting:\ line\ bundles\ and\ linear\ series.\ MR2095472\ (2005k:14001a)
\bibitem[Laz04b]{L2} R.\ Lazarsfeld:\ \emph{Positivity in algebraic geometry. II},\ Ergebnisse\ der\ Mathematik\ und\ ihrer\ Grenzgebiete\. 3.\ Folge.\ A\ Series\ of\ Modern\ Surveys\ in\ Mathematics\ [Results\ in\ Mathematics\ and\ Related\ Areas.\ 3rd\ Series.\ A\ Series\ of\ Modern\ Surveys\ in\ Mathematics],\ vol.\ 49,\ Springer-Verlag,\ Berlin,\ 2004,\ Positivity\ for\ vector\ bundles,\ and\ multiplier\ ideals.\ MR2095472\ (2005k:14001b)

\bibitem[Mor85]{M}S.\ Mori:\ \emph{Classification of Higher-Dimensional Varieties}, \ Algebraic\ Geometry\ (Brunswick,\ Maine,\ 1985),\ Proc.\ Sympos.\ Pure\ Math.\ {\bf 46},\ Amer.\ Math.\ Soc.,\ Providence,\ 1987,\ 269-331.

\bibitem[Muk81]{Mu} S.\ Mukai:\ \emph{Duality between $D(X)$ and $D(\hat{X})$ with its application to Picard sheaves},\ Nagoya\ math.\ J.\ {\bf 81},\ 1981,\ 153-175.

\bibitem[Sim93]{Sim} C.\ Simpson:\ \emph{Subspaces of moduli spaces of rank one local systems}\ Ann.\ scient.\ \'Ecole\ Norm.\ Sup.\ 1993,\ 26,\ 361-401.

\bibitem[Siu09]{S} Y.T.\ Siu:\ \emph{Abundance conjecture},\ 2009.\ arXZiv:math.AG/0912.0576

\bibitem[Tak08]{Tak} S.\ Takayama:\ \emph{On uniruled degenerations of algebraic varieties with trivial canonical divisor},\ Math.\ Z.\ (2008)\ 259:487-501.

\bibitem[Uen75]{U} K.\ Ueno:\ \emph{Classification Theory of Algebraic Varieties and Compact Complex Spaces},\ Springer\ Lecture\ Note\ 439,\ 1975.

\bibitem[Vie83]{V1} E.\ Viehweg:\ \emph{Weak positivity and the additivity of the Kodaira dimension of certain fiber spaces}.\ Adv.\ Studies\ in\ Pure\ Math\ 1,\ 1983,\ 329-353.

\bibitem[Vie95]{V2} E.\ Viehweg:\ \emph{Quasi-projective moduli for polarized manifolds},\ Ergenbisse\ der\ Mathematik\ 30\ (Springer,\ 1995).

\end{thebibliography}
\end{document}